\documentclass[10pt]{amsproc}
\usepackage{amsmath, amsthm,amssymb}

\usepackage[pagebackref, colorlinks, linkcolor=red, citecolor=blue, urlcolor=blue, hypertexnames=true]{hyperref}
\usepackage{setspace}
\usepackage{amsrefs}

\allowdisplaybreaks

\title{The expected degree of minimal spanning forests}

\author{Andreas Thom}
\address{A.T., Mathematisches Institut, U Leipzig,
PF 100920, 04009 Leipzig, Germany}
\email{andreas.thom@math.uni-leipzig.de}

\date{\today}

\newtheorem{thm}{Theorem}

\newtheorem{cor}[thm]{Corollary}
\newtheorem{lem}[thm]{Lemma}
\theoremstyle{definition}

\newtheorem{rem}[thm]{Remark}

\newcommand{\N}{{\mathbb N}}
\newcommand{\Z}{{\mathbb Z}}
\newcommand{\R}{{\mathbb R}}

\begin{document}

\onehalfspace

\maketitle
\begin{abstract} 
We give a lower bound on the expected degree of the free minimal spanning forest of a vertex transitive graph in terms of
its spectral radius. This result answers a question of Lyons-Peres-Schramm and simplifies the Gaboriau-Lyons proof of the measurable-group-theoretic solution to von Neumann's problem.

In the second part we study a relative version of the free minimal spanning forest. As a consequence of this study we can show that non-torsion unitarizable groups have fixed price one.
\end{abstract}

%\tableofcontents
\section{Free minimal spanning forests}
\label{FMSF}
Let $G=(V,E^\sharp)$ be a connected graph. We view $E^\sharp$ as a set of oriented edges with the requirement, that for each oriented edge in $E^\sharp$, the opposite edge is also contained in $E^\sharp$. For each oriented edge $e$, we denote by $\hat e$ the opposite edge, by $\underline{e}$ the tail and by $\overline{e}$ the head of $e$. We require $e \neq \hat e$ for all $e \in E^\sharp$. 
Thus, an {\it edge} is naturally identified with a unordered pair of oriented edges $\{e,\hat e\}$. The set of edges is denoted by $E$. We say that $G$ is vertex-transitive if for any two vertices $x,y \in V$, there exists an automorphism $\varphi$ of $G$, such that $\varphi(x)=y$. From now on, we will only consider vertex-transitive graphs. Let $d$ be the degree of a vertex in $G$.

A path of length $n$ between vertices $x,y \in V$ is a sequence of oriented edges $(e_1,\dots,e_n)$ such that $\overline{e}_i = \underline{e}_{i+1}$ for all $1 \leq i \leq n$, $\underline{e}_1=x$ and $\overline{e}_n=y$. A path is called a cycle rooted at $x \in V$ if  $x=y$. A cycle is called simple if no vertex is visited twice  and it is not of the form $(e,\hat e)$.
We denote the number of cycles rooted at $x$ of length $n$ by $c(n)$. The spectral radius of $G$ is defined to be
$$\lambda := \frac1d \cdot \limsup_{n \to \infty} c(n)^{1/n}$$
and it equals the operator norm of the random walk operator on $\ell^2(V)$, see \cite[Lemma 10.1]{woess}. For $x,y \in V$, we denote by $c(n,x,y)$ the number of paths of length $n$ from $x$ to $y$. A basic consequence of the description of the operator norm of the random walk operator is the estimate
\begin{equation} \label{est1}
c(n,x,y) \leq (\lambda d)^n.
\end{equation}

A subgraph of $G=(V,E^\sharp)$ is a subset $V_1 \subset V$ and a subset $E_1^\sharp \subset E^\sharp$, such that $G_1=(V_1,E_1^\sharp)$ is itself a graph, i.e. $E_1^\sharp$ is closed under taking opposite oriented edges and oriented edges in $E_1^\sharp$ connect only points in $V_1$. A subgraph $G_1=(V_1,E_1^\sharp)$ is called spanning if $V_1=V$. A subgraph of $G$ is called a forest, if it does not contain any simple cycles.

Let us start by recalling the definition of the minimal spanning forest of $G$. First of all, a spanning subgraph of $G$ is naturally identified with an element of the space $\{0,1\}^{E}$. 
The minimal spanning forest is a probability measure $\rm FMSF$ on $\{0,1\}^{E}$ whose support consists of spanning forests. The measure $\rm FMSF$ arises as the push-forward of the product measure $\mu^{\otimes E}$ along a map
$$\Phi \colon [0,1]^{E} \to \{0,1\}^{E},$$
where $\mu$ denote the Lebesgue measure on $[0,1]$.
To define $\Phi$, we view an element $x \in [0,1]^{E}$ as a $[0,1]$-labelling of the edges of $G$. Whenever we see a simple cycle, then cut all edges in this cycle, which carry a maximal label. The resulting graph is denoted by $\Phi(x)$ and is clearly a spanning forest. It can be checked directly, that $\Phi$ is a Borel map. The study of minimal spanning forests on lattices in $\R^d$ and specific graphs has a long history. The systematic study of ${\rm FMSF}$ on unimodular vertex-transitive graphs was initiated by Lyons-Peres-Schramm \cite{MR2271476}, and we refer to this paper also for further references. 
The virtue of the construction from above is that the resulting probability measure on $\{0,1\}^E$ is invariant under all automorphisms of the graph $G$. In particular, since we assume $G$ to be vertex-transitive, what happens at a single vertex happens everywhere. 
Let $x \in V$ be some fixed vertex and let $\delta_{\rm FMSF}(G)$ be the expected degree of the vertex $x$ with respect to $\rm FMSF$. The number $\delta_{\rm FMSF}(G)$ (which is clearly independent of $x$) is called the expected degree of the free minimal spanning forest. 

\begin{thm} \label{main} Let $G$ be a vertex transitive graph with degree $d$ and spectral radius $\lambda$. Then, the following inequalities hold:
$$  \frac1{4 \lambda} - \frac14 \leq \delta_{\rm FMSF}(G) \leq d.$$
\end{thm}
\begin{proof} The second inequality is clear.
Let us enumerate the edges at the vertex $x$ by $s_1,\dots,s_d$. For $1 \leq i \leq d$, the probability that $s_i$ survives is zero if $s_i$ is a loop. Let us assume that $s_i$ is not a loop and that its label is $\rho$. Then, the probability that it is cut because of a simple cycle of length $n$ rooted at $x$ and starting with $s_i$ is $\rho^{n-1}$. Moreover, by \eqref{est1} there are at most $(\lambda d)^{n-1}$ such cycles.
Hence, the probability that $s_i$ survives is bounded from below as follows:
\begin{eqnarray*}
{\mathbb P}(s_i \mbox{ survives}) &\geq& \int_{0}^{\frac1{2 \lambda d}} \left(1 - \sum_{n=2}^{\infty} (\lambda d)^{n-1} \rho^{n-1} \right)\ d\rho \\
&=& \int_{0}^{\frac1{2 \lambda d}} \left(2 - \frac{1}{1 - \lambda d \rho}\right) \ d\rho\\
&=& \frac{1 - \ln(2)}{ \lambda d}\\
&\geq& \frac{1}{4 \lambda d}.
\end{eqnarray*} 
%for any $\eta \in [0,1]$.
%Setting $\eta = \frac1{3\lambda d}$, we obtain $${\mathbb P}(s_i \mbox{ survives}) \geq \frac{1}{3\lambda d} \left( 1 -\frac12 \cdot \frac{\frac13}{1 - \frac13}\right) = \frac{1}{4 \lambda d}.$$ 
Since this argument works for all edges adjacent to the vertex $x$ which are not loops, and the number of loops is bounded from above by $\lambda d$, we obtain
$$\delta_{\rm FMSF}(G) =  {\mathbb E} \left(\sum_{i=1}^d \ [s_i \mbox{ survives}] \right) =  \sum_{i=1}^d{\mathbb P} \left(s_i \mbox{ survives} \right) \\
\geq \frac{d - \lambda d}{4 \lambda d} = \frac1{4 \lambda} - \frac14.$$ Here, we denoted by $[s_i \mbox{ survives}]$ the $\{0,1\}$-valued function on $[0,1]^E$, which indicates if $s_i$ survived the cutting procedure or not.  This proves the claim.
\end{proof}

Let $\Gamma$ be a finitely generated group and $S$ a finite multi-set with $S=S^{-1}$ which generates $\Gamma$. More formally, a multi-set is a map $\alpha\colon S \to \Gamma$ and $S^{-1}=S$ means that there is an involution $\iota \colon S \to S$, such that $\alpha(\iota(s))=\alpha(s)^{-1}$. We will omit $\alpha$ and $\iota$ in the notation and just write $s^{-1}$ for the inverse.

Let ${\rm Cay}(\Gamma,S)$ be the associated Cayley graph, i.e.\ the vertex set is $\Gamma$ and for every $g \in \Gamma$ and $s \in S$, we have an oriented edge $(g,s,sg)$. The opposite edge is given by $(sg,s^{-1},g)$. We denote by $\delta_{\rm FMSF}(\Gamma,S)$ the expected degree of the free minimal spanning forest of the Cayley graph ${\rm Cay}(\Gamma,S)$.

\begin{cor} \label{corol}
Let $\Gamma$ be a finitely generated non-amenable group. The expected degree of the free minimal spanning forest depends on the generating multi-set. Moreover, for every $n \in \N$, there exists a finite multi-set $S$ which generates the group $\Gamma$, such that $\delta_{\rm FMSF}(\Gamma,S) \geq n$.
\end{cor}
\begin{proof}
Since $\Gamma$ is non-amenable, $\lambda < 1$ for any finite generating set $S$ by Kesten's theorem, see \cite{kesten}. Replacing $S$ by the multi-set $S^{[k]}$ of words of length $k$ in $S$, then we get a spectral radius $\lambda^k$ and thus from Theorem \ref{main}:
$$\frac1{4 \lambda^k} - \frac14 \leq \delta_{\rm FMSF}(\Gamma,S^{[k]}) \leq d^k.$$ In particular, $\delta_{\rm FMSF}(\Gamma,S^{[k]})$ cannot be independent of $k$ and is unbounded as $k$ tends to infinity.
\end{proof}

The previous corollary answers Question 6.13 in \cite{MR2271476}, where it was asked whether the expected degree of $\rm FMSF$ depends on the choice of a generating set or not. Russ Lyons informed us that as a solution to  \cite[Exercise 11.16]{lyons}, it was known that dependence on the generating set holds for some specific non-amenable groups.
As a consequence of Corollary \ref{corol}, we can now conclude that ${\rm WMSF}(\Gamma,S) \neq {\rm FMSF}(\Gamma,S)$ for any non-amenable group and some suitable multi-set of generators $S$, see \cite{MR2271476} for more definitions. As an immediate consequence of \cite[Proposition 3.6]{MR2271476} and Corollary \ref{corol}, we also get that $p_c(\Gamma,S) < p_u(\Gamma,S)$ for any non-amenable group and a suitable multi-set of generators $S$. This result was shown by Pak-Smirnova-Nagnibeda as \cite[Theorem 1]{MR1756965} -- and we refer to this paper for definitions if necessary. Previously, \cite[Theorem 1]{MR1756965} was used to conclude that ${\rm WMSF}(\Gamma,S) \neq {\rm FMSF}(\Gamma,S)$ via \cite[Proposition 3.6]{MR2271476}.
Another consequence of Corollary \ref{corol} is an  elementary proof of part of \cite[Proposition 12]{MR2534099}, which is the key new technical result in the Lyons-Gaboriau proof of the measurable-group theoretic solution to von Neumann's problem. Note that it does not follow from our results that the cluster equivalence relation of ${\rm FMSF}$ is ergodic, but for the purposes of a proof of \cite[Theorem 1]{MR2534099} this can be fixed by a result of Chifan-Ioana \cite[Corollary 9]{ioana}, see also the remarks in \cite[Section 4.2]{ioana}. 

Let us also mention, that it was proved by Timar in \cite{MR2243871}, that if ${\rm WMSF}(\Gamma,S) \neq {\rm FMSF}(\Gamma,S)$, then almost surely every tree in ${\rm FMSF}$ has infinitely many ends.

\section{Relative minimal spanning forests}

Let $\Gamma$ be a non-torsion group and let $a \in \Gamma$ be a non-torsion element. Let $S \subset \Gamma$ be a finite generating set with $S^{-1}=S$ and consider the Cayley graph $G:={\rm Cay}(\Gamma,S)$. We set $d:=|S|$. We will assume that $a \in S$. If this is the case, then for any $g \in \Gamma$ the set $L := \{a^ng \mid n \in \Z\} \subset \Gamma$ is a bi-infinite line in $G$. We denote by $L(a)$ the subgraph of $G$ formed by the union of all such bi-infinite lines.

Note that $\Gamma$ acts by automorphisms on $G=(\Gamma,E)$.
An element of $ x \in \{0,1\}^E$ is identified with a subgraph $G(x) \subseteq G$. A random spanning sub-forest of $G$ is a probability measure $\sigma$ on $\{0,1\}^{E}$, whose support is contained in the set of spanning forests. We call such a random subforest invariant, if the probability measure is invariant with respect to the natural $\Gamma$ action on $E$. An invariant random subforest $\sigma$ is called a factor of i.i.d.\ process if there exists a measurable and $\Gamma$-equivariant map $\Phi \colon [0,1]^{E} \to \{0,1\}^E$, such that $\sigma = \Phi_*(\mu^{\otimes E})$, where $\mu$ denotes the Lebesgue measure on $[0,1]$. If $\sigma$ is an invariant random subforest of $G$, then we denote by ${\rm deg}(\sigma)$ the expected degree of the vertex $e \in \Gamma$.

We will not recall the definition of cost of a p.m.p.\ essentially free action here and instead refer to Gaboriau's foundational work \cite{gaboriau} and the book \cite{kechrismiller}, where everything is explained in detail. The cost of the action $\Gamma \curvearrowright (X,\lambda)$ is denoted by ${\rm cost}(\Gamma \curvearrowright X)$.

The following theorem deals with a relative version of the free minimal spanning forest, paying at the same time more attention to a control of the expected degree in terms of the cost of the action $\Gamma \curvearrowright [0,1]^E$. Results of this type have been studied before, see e.g.\ \cite[Lemma 28.11]{kechrismiller} or \cite[Corollary 40]{pichot}.

\begin{thm} \label{relmin}
Let $\Gamma$ be a non-torsion group and let $a \in \Gamma$ be a non-torsion element. Let $S$ be a finite symmetric generating set with $a \in S$ and denote the associated Cayley graph by $G$. There exists a random spanning sub-forest $\sigma$ of $G$, such that
\begin{enumerate}
\item $\sigma$-a.s. $L(a)$ is contained in the subforest, 
\item $\sigma$ is a factor of i.i.d.\ process, and
\item ${\rm deg}(\sigma) \geq 2 \cdot {\rm cost}(\Gamma \curvearrowright [0,1]^E)$.
\end{enumerate}
\end{thm}
\begin{proof} The proof follows the ideas of the proof of \cite[Lemma 28.11]{kechrismiller}. First of all, we pick an isomorphism $[0,1] \cong [0,1]^\N$ and think of $[0,1]^{\N}$ as an infinite stack of elements in the unit interval. Let $n \in \N$. We first define the graph $Z_n(G)$ of $n$-cycles in $G$. Its vertex set is the set $C_n$ of (unrooted and simple) $n$-cycles in $G$ and two cycles are connected by an edge if they are different and have a common vertex. It is clear that the maximum degree of $Z_n(G)$ is less or equal $n d^n$, where we have set $d:=|S|$. Using the $[0,1]$-labels of the edges of $G$, we construct an i.i.d.\ labelling of $C_n$ with elements in $[0,1]$ in an equivariant way. Now, let $c_0 \in C_n$ labelled $\lambda \in [0,1]$. The probability that there is a chain $c_0,c_1,\dots,c_k$ of length $k$ of pairwise adjacent cycles so that their labels are increasing in each step is bounded above by
$(1- \lambda^{nd^n})^k$. We see from this, that the probability of existence of an infinite such chain is zero. Thus -- for almost every labelling of $G$ -- we can assign to each $n$-cycle the length of the longest such chain, and call it the depth of this $n$-cycle (with respect to the given labelling of $G$).

In this way, we have for almost every labelling $x \in [0,1]^E$, found a map
$$\varphi_x \colon \bigsqcup_{n \geq 1} C_n \to \N \times \N,$$
which maps an $n$-cycle of depth $k$ (with respect to the labelling $x \in [0,1]^E$) to the pair $(n,k) \in \N \times \N$. We call this map a colouring of the cycles associated with $x \in [0,1]^E$. Note that this colouring is equivariant in the sense that $\varphi_{xg}(c) = \varphi_x(cg^{-1})$, for almost all $x \in [0,1]^E$ and all $c \in \sqcup_{n \geq 1} C_n$.

Thus, we have used the first element in our stack $[0,1]^\N$ to set up an $\Gamma$-equivariant colouring of the set of cycles. We now will use the other elements in the stack to set up an infinite recursive $\Gamma$-equivariant cutting procedure which will have the property, that it does not disconnect the graph $G$. Let's pick a standard enumeration $\alpha \colon \N \times \N \cong \N$ and set $\psi_x:=\alpha \circ \varphi_x$. We start by cutting all cycles in $\psi_x^{-1}(0)$ in a way yet to be described. Note that the set $\psi_x^{-1}(0)$ consists of disjoint cycles. For each $c \in \psi_x^{-1}(0)$, we cut the edge with the maximal label (this is the second in our stack $[0,1]^\N$) not in $L(a)$. Clearly, this does not disconnect $G$. We now proceed in a similar way (using the next element in the stack) with the cycles in $\psi^{-1}_x(1)$ -- again, this is a set of disjoint cycles -- which still exist after the cutting that has been done already, etc.

In each step, the graph remains connected, and hence the expected degree cannot be less than $2 \cdot {\rm cost}(\Gamma \curvearrowright [0,1]^E)$. Moreover, the resulting graphs will always contain $L(a)$. Taking limits, we see that all cycles of $G$ have been cut and we obtain a random sub-forest of $G$ containing $L(a)$, whose expected degree is at least ${\rm cost}(\Gamma \curvearrowright [0,1]^E)$. We can now define $\Phi \colon [0,1]^E \to \{0,1\}^E$ by mapping $x \in [0,1]^E$ to the result of the cutting process. This finishes the proof.
\end{proof}

\begin{rem} \label{remabert}
Let $\Gamma$ be a finitely generated and infinite group. Recall, an infinite group is said to have fixed price one if all its p.m.p.\ essentially free actions have cost one.
Then, ${\rm cost}(\Gamma \curvearrowright [0,1]^E)>1$ if and only if $\Gamma$ does not have fixed price equal to one, as proved by Ab\'ert-Weiss. 
Indeed, by the Ab\'ert-Weiss theorem, ${\rm cost}(\Gamma \curvearrowright [0,1]^E)$ is the maximum among costs of all possible p.m.p.\ essentially free actions of the group $\Gamma$. Hence, if this cost is equal to one, the group must have fixed price one. The other implication is obvious.
\end{rem}

We now change the perspective slightly and consider more general invariant random spanning forests. They are no longer bound to be sub-forests of $G$. Thus, we study more generally probability measures on $\{0,1\}^{\Gamma \times \Gamma}$, invariant under the diagonal right $\Gamma$-action on $\Gamma \times \Gamma$, see \cite{MR2552305} for more details. The expected degree ${\rm deg}(\sigma)$ of an invariant random spanning forest $\sigma$ defined in a similar way as before. The width of an invariant random spanning forest $\sigma$ ist defined to be the number of vertices $g \in \Gamma$, so that the probability that an edge between vertices $e$ and $g$ exists is positive. We denote this number by ${\rm width}(\sigma)$.

Let us now assume that there is some invariant random sub-forest $\tau$ of $G$, such that ${\rm deg}(\tau)>2$ and $L(a)$ is a.s.\ contained in the sub-forest. Let $b \in S \setminus \{a,a^{-1}\}$, such that
$\tau( \{x \in \{0,1\}^E \mid (b, e) \in G(x) \} )>0$ and set $$X:= \left\{x \in \{0,1\}^E \mid (b, e) \in G(x) \right\} \subset \{0,1\}^E.$$ For each $n \in \N$, we now define a Borel map $\Theta_n \colon \{0,1\}^E \to \{0,1\}^{\Gamma \times \Gamma}$ which sends a subgraph $G(x) \subseteq G$  to the graph formed by all edges $(a^ib a^{-i}g,g)$ for $1 \leq i \leq n$, whenever $a^ib a^{-i}g$ and $g$ lie in the same connected component of $G(x)$. It is clear that $\Theta_n$ is $\Gamma$-equivariant. We define $\tau_n$ to be the push-forward of $\tau$ with respect to $\Theta_n$, i.e. $\tau_n:= (\Theta_n)_*(\tau)$.

\begin{lem} \label{comp}
For each $n \in \N$, the measure $\tau_n$ is an invariant random spanning forest. Moreover, we have ${\rm width}(\tau_n) = 2n$ and ${\rm deg}(\tau_n) = 2n \cdot \tau(X)$.
\end{lem}
\begin{proof} Consider the free group on letters $\{a,b\}$.
It is a well-known fact that the elements $\{a^i b a^{-i} \mid 1 \leq i \leq n \}$ are the basis of a free group of rank $n$. Thus any cycle formed by the partial self-maps $a^iba^{-i}|_{a^i(X)}$ and their inverses yields a cycle formed by the partial self-maps $a,a^{-1},b|_X, b^{-1}|_{b(X)}$. However, as $\tau$ does not contain cycles, there are no such cycles. Invariance is clear since $\Theta_n$ was $\Gamma$-equivariant.

It is clear that the width of $\tau_n$ is equal to $2n$, whereas the expected degree is equal to $2n \cdot \tau(X)$.
\end{proof}

We can now obtain a corollary which covers a particular case of the famous Dixmier problem on unitarizability, see the book of Pisier \cite{MR1818047} for a detailed discussion of this problem and further references. Our approach relies on results of Epstein-Monod \cite{MR2552305}, whose pioneering work related invariant random forests to this problem. 
Let us explain this in a bit more detail. Recall, a representation $\pi \colon \Gamma \to B(\mathcal H)^{\times}$ on a Hilbert space is called uniformly bounded if $\sup_{g \in \Gamma} \|\pi(g)\|< \infty$. It was shown by Dixmier that any uniformly bounded representation of an amenable group is conjugate to a unitary representation. He conjectured that this in fact yields a characterization of amenability. Groups with the property that every uniformly bounded representation on a Hilbert space is conjugate to a unitary representation (i.e.\ can be unitarized) are called {\it unitarizable}. It is well-known by now that non-abelian free groups (and all groups containing such groups) are not unitarizable. However, there are non-amenable groups without non-abelian free subgroups. After the results of Gaboriau-Lyons \cite{MR2534099}, there was some hope that Dixmier's longstanding problem could be resolved, using suitable invariant spanning forests that play the role of free subgroups, see \cite{MR2552305, ozawa}. We are now building on the work of Epstein-Monod \cite{MR2552305} and construct a specific sequence of invariant random forests showing that certain groups are not unitarizable.

\begin{cor}
Let $\Gamma$ be a finitely generated non-torsion group which does not have fixed price one. Then, the group $\Gamma$ is not unitarizable.
\end{cor}
\begin{proof} Let $a \in \Gamma$ be a non-torsion element and let $S \subset \Gamma$ be a finite generating set with $a \in S$.
Consider the Cayley graph $G$ associated with $S$ as above. By Remark \ref{remabert}, if $\Gamma$ does not have fixed price one, then ${\rm cost}(\Gamma \curvearrowright [0,1]^E)>1$. By  Theorem \ref{relmin}, there exists an invariant random spanning forest $\tau$ with ${\rm deg}(\tau)>2$ and $L(a)$ contained in it almost surely. 

Thus, we can follow the construction of $(\tau_n)_n$ as above.
In particular, we conclude from Lemma \ref{comp} that
\begin{equation} \label{unb} \frac{{\rm deg}(\tau_n)^2}{{\rm width}(\tau_n)} = 2n \cdot \tau(X)^2\to \infty, \quad \mbox{as}\quad  n \to \infty.\end{equation}
Now, the desired result is an immediate consequence of \eqref{unb} and \cite[Theorem 1.3]{MR2552305} from the work of Epstein-Monod.
\end{proof}

\begin{rem}
The preceding result can also be proved along the same lines under the assumption that $\Gamma$ is finitely generated, does not have fixed price one, and that there is no bound on the order of finite subgroups of $\Gamma$ or that $\Gamma$ contains an infinite amenable subgroup. This covers various finitely generated simple torsion groups with positive first $\ell^2$-Betti number, see \cite{thomosin} for examples.
\end{rem}

\section*{Acknowledgments}

I want to thank Russ Lyons for interesting comments on Section \ref{FMSF}, pointing out \cite[Exercise 11.16]{lyons} in his book, and explaining some parts of \cite{MR2534099}. I am grateful to Art\"em Sapozhnikov and Nicolas Monod for interesting and motivating discussions. Most of the results presented in this note have been obtained during the Hausdorff Trimester Program on {\it Rigidity} in 2009. The author is grateful to this institution for its hospitality and to Wolfgang L\"uck and Nicolas Monod for organizing this event. This research was supported by ERC-Grant No.\ 277728 "Geometry and Analysis of Group Rings".

%\section{Lines in Cayley graphs}
%
%Consider a non-amenable group $\Gamma$, generated by elements $a,b \in \Gamma$. Assume that $a$ is non-torsion and $b^2 \neq e$. We set $S:= \{a,a^{-1},b,b^{-1}\}$ and consider the Cayley graph $G:= {\rm Cay}(\Gamma,S)$. 
%We denote by $b(k)$ the number of elements in the ball of radius $k$ in the Cayley graph $G$. The group $\Gamma$ is partitioned into cosets of the form $\{a^n g \mid n \in \Z \}$, which we call lines.
%We will assume that $\Lambda := \langle a \rangle \subset \Gamma$ is malnormal, i.e.\ $g\Lambda g^{-1} \cap \Lambda = \{e\}$ for all $g \not \in \Lambda$. 
%
%\begin{lem} Assume that $\Lambda \subset \Gamma$ is malnormal.
%Let $L,L'$ be a pair of lines and let $k \in \N$. Then,
%$$|\{(s,t) \in L \times L' \mid d(s,t) \leq k\}| \leq b(k).$$
%\end{lem}
%\begin{proof} First of all, choose $g,h \in \Gamma$, so that $L = \Lambda g$ and $L'= \Lambda h$. Note that
%$d(a^ng, a^mh ) = \ell(a^n gh^{-1} a^{-m})$. Hence, if the required bound fails, then $a^n gh^{-1} a^{-m} = a^{n'} gh^{-1} a^{-m'}$ for two pairs $(n,m), (n',m')$. If $n=n'$, then $m=m'$. Thus, we may assume that $n \neq n'$ and $m\neq m'$. We get
%$$a^{n-n'} = (gh^{-1})a^{m-m'}(gh^{-1})^{-1}.$$ Thus, since $\Lambda$ is malnormal, we conclude $gh^{-1} \in \Lambda$. This implies $L=L'$, a contradiction.
%\end{proof}

\end{document}